\documentclass[a4paper,12pt]{amsart}
\usepackage{amsmath}
\usepackage{amssymb}
\usepackage{enumerate}
\usepackage{amscd}
\usepackage{graphics}
\usepackage{latexsym}
\usepackage{verbatim} 
\usepackage{url}
\usepackage{stmaryrd}
\usepackage[foot]{amsaddr}

\theoremstyle{plain}
\newtheorem{thm}{{\bf Theorem}}[section]

\newtheorem{lemma}[thm]{{\bf Lemma}}
\newtheorem{fact}[thm]{{\bf Fact}}
\newtheorem{claim}[thm]{{\bf Claim}}

\theoremstyle{definition}
\newtheorem{define}[thm]{{\bf Definition}}

\newtheorem{question}[thm]{{\bf Question}}

\newcommand{\size}[1]{\left\vert {#1} \right\vert}
\newcommand{\p}{\mathcal{P}}

\newcommand{\seq}[1]{\langle {#1} \rangle}

\newcommand{\ka}{\kappa}
\newcommand{\la}{\lambda}
\newcommand{\om}{\omega}

\newcommand{\bbM}{\mathbb{M}}
\newcommand{\bbP}{\mathbb{P}}

\newcommand{\ZF}{\mathsf{ZF}}
\newcommand{\ZFC}{\mathsf{ZFC}}

\title[]{Extendible cardinals and the mantle}
\author[T. Usuba]{Toshimichi Usuba}\thanks{This research was supported by
JSPS KAKENHI grant Nos. 18K03403 and 18K03404.}
\address[T. Usuba]
{Faculty of Science and Engineering,
Waseda University, 
Okubo 3-4-1, Shinjyuku, Tokyo, 169-8555 Japan.
Tel: +81 3 5286 3000}
\email{usuba@waseda.jp}

\begin{document}
\maketitle
\begin{abstract}
The mantle is the intersection of all ground models of $V$.
We show that if there exists an extendible cardinal
then the mantle is the smallest ground model of $V$.
\end{abstract}

\section{Introduction}

Let us say that an inner model $W$ of $\ZFC$ is a \emph{ground}
if the universe $V$ is a set-forcing extension of $W$.
The \emph{set-theoretic geology}, initiated by Fuchs-Hamkins-Reitz \cite{FHR},
is a study of the structure of all grounds of $V$.
An important object in the set-theoretic geology is the mantle:
\begin{define}
The mantle $\bbM$ is the intersection of all grounds of $V$.
\end{define}
It is known that the mantle is a transitive, definable, forcing-invariant model of $\ZFC$ (\cite{FHR}, Usuba \cite{Usuba}).
The mantle itself needs not to be a ground of $V$, but 
if the mantle is a ground  of $V$ then it is the smallest ground.
In \cite{FHR}, they asked the following question:
Under what circumstances is the mantle also a ground model of the universe?
For this question, Usuba answered  if
some very large cardinal exists then the mantle must be a ground:

\begin{fact}[\cite{Usuba}]
Suppose there exists a hyper-huge cardinal.
Then the mantle is a ground of $V$.
More precisely, there is a poset $\bbP \in \bbM$ and
a $(\bbM ,\bbP)$-generic $G$ such that $\size{\bbP}<\ka$ and
$V=\bbM[G]$.
\end{fact}
An uncountable cardinal $\ka$ is \emph{hyper-huge}
if for every cardinal $\la \ge \ka$,
there is an elementary embedding $j:V \to M$ for some inner model $M$
such that the critical point of $j$ is $\ka$, $\la<j(\ka)$, and
$M$ is closed under $j(\la)$-sequences.

In this paper we prove that the hyper-huge cardinal assumption can be weakened
to the extendible cardinal assumption. 
Recall that an uncountable cardinal $\ka$ is \emph{extendible}
if for every ordinal $\alpha \ge \ka$, there exists $\beta>\alpha$ and
an elementary embedding $j:V_\alpha \to V_\beta$ such that
the critical point of $j$ is $\ka$ and $\alpha<j(\ka)$.
Every hyper-huge cardinal is an extendible cardinal  limit of extendible cardinals.
\begin{thm}\label{mainthm}
Suppose there exists an extendible cardinal.
Then the mantle is a ground of $V$.
In fact if $\ka$ is extendible then the $\ka$-mantle of $V$ is its smallest ground
(The $\ka$-mantle will be defined in Definition \ref{ka-ground} below).
\end{thm}

\section{Some materials}
Let's recall some basic definitions and facts about the set-theoretic geology.
See \cite{FHR} for more information.

\begin{fact}[\cite{FHR}, Reitz \cite{Reitz}]
\label{definability}
There is a formula $\varphi(x,y)$ of set-theory such that:
\begin{enumerate}
\item For every $r$, the class $W_r=\{x \mid \varphi(x,r)\}$ is a ground of $V$ with
$r \in W_r$.
\item For every ground $W$ of $V$, there is $r$ with
$W=W_r$.
\end{enumerate}
\end{fact}

It turned out that all grounds are downward set-directed:

\begin{fact}[\cite{Usuba}]\label{DDG}
Let $\{W_r \mid r \in V\}$ be the collection of all grounds of $V$ defined as in Fact 
\ref{definability}.
For every set $X$, there is $r$ such that
$W_r \subseteq W_s$ for every $s \in X$.
\end{fact}

A key of the definability of grounds as in Fact \ref{definability} is the covering and the approximation properties introduced by Hamkins \cite{Hamkins}:
\begin{define}[\cite{Hamkins}]
Let $M \subseteq V$ be a transitive model of $\ZFC$ containing all ordinals.
Let $\ka$ be a cardinal.
\begin{enumerate}
\item $M$ satisfies the \emph{$\ka$-covering property} for $V$ if
for every set $x$ of ordinals, if $\size{x}<\ka$ then there is $y \in M$ with
$x \subseteq y$ and $\size{y}<\ka$.
\item $M$ satisfies the \emph{$\ka$-approximation property} for $V$ if
for every set $A$ of ordinals, if $A \cap x \in M$ for every set $x \in M$ with size $<\ka$,
then $A \in M$.
\end{enumerate}
\end{define}

\begin{fact}[Hamkins, see Laver \cite{Laver}]\label{uniqueness}
Let $\ka$ be a regular uncountable cardinal.
Let $M, N$ be transitive models of $\ZFC$ containing all ordinals.
If $M$ and $N$ satisfy the $\ka$-covering and the $\ka$-approximation properties for $V$,
$\p(\ka) \cap M=\p(\ka) \cap N$, and $\ka^+=(\ka^+)^M=(\ka^+)^N$,
then $M=N$.
\end{fact}

\begin{fact}[\cite{Hamkins}]
\label{cov. and approx. of forcing ext.} 
Let $\ka$ be a regular uncountable cardinal, and $\bbP$ a poset of size $<\ka$.
Let $G$ be $(V, \bbP)$-generic.
Then, in $V[G]$, $V$ satisfies the $\ka$-covering and the $\ka$-approximation properties for $V[G]$.
\end{fact}

Let us make some definition and observations.

\begin{define}\label{ka-ground}
Let $\ka$ be a cardinal.
A ground $W$ of $V$ is a \emph{$\ka$-ground} if
 there is a poset $\bbP \in W$ of size $<\ka$
and a $(W,\bbP)$-generic $G$ such that
$V=W[G]$.
The \emph{$\ka$-mantle} is the intersection of all $\ka$-grounds.
\end{define}
The $\ka$-mantle is a definable, transitive, and extensional class.
It is trivially consistent that the $\ka$-mantle is a model of $\ZFC$,
and we can prove that if $\ka$ is strong limit, then the $\ka$-mantle must be  a model of $\ZF$.
A sketch of the proof is as follows.
First we show that all $\ka$-grounds are downward-directed.
For any two $\ka$-grounds $W_0$ and $W_1$,
since $\ka$ is a limit cardinal,
there is a regular cardinal $\la<\ka$ such that
$W_0$ and $W_1$ are $\la$-grounds.
Then $V$ is a $\la$-c.c. forcing extension of both $W_0$ and $W_1$.
By the proof of Fact \ref{DDG} (see \cite{Usuba}),
we can find a ground $W \subseteq W_0 \cap W_1$ of $V$ such that
$V$ is a $\la^{++}$-c.c. forcing extension of $W$.
Then, we can find a poset $\bbP \in W$ of size $\le 2^{2^{{\la}^{++}}}$
and a $(W, \bbP)$-generic filter $G$ with $V=W[G]$
(e.g., see Appendix in Sargsyan-Schindler \cite{SS}), so $W$ is a $\ka$-ground of $V$ as well.
The downward-directedness of the $\ka$-grounds implies that
the $\ka$-mantle is absolute between all $\ka$-grounds,
so the $\ka$-mantle of $V$ is definable in all $\ka$-grounds.
Now, by Lemma 21 in \cite{FHR}, the $\ka$-mantle is a model of $\ZF$.

However we do not know whether the $\ka$-mantle is always a model of $\ZFC$.
\begin{question}
For a given cardinal $\ka$,
is the $\ka$-mantle always a model of $\ZFC$?
\end{question}

By Facts \ref{uniqueness} and \ref{cov. and approx. of forcing ext.},
if $W$ is a $\ka$-ground of $V$,
then $W$ is completely determined by the set $P=\p(\ka) \cap W$,
that is, $W$ is a  unique ground $W'$ with
$\p(\ka) \cap W'=P$, $\ka^+=(\ka^+)^{W'}$, and
$W'$ satisfies the $\ka$-covering and the $\ka$-approximation properties.
This means that there are at most $2^{2^\ka}$ many $\ka$-grounds of $V$,
hence there is a set $X$ of size $\le 2^{2^\ka}$
such that 
the collection $\{W_r \mid r \in X\}$ is the $\ka$-grounds.
We have the following by the combination of this observation and Fact \ref{DDG}:
\begin{lemma}\label{2.6}
Let $\ka$ be a cardinal and $\overline{W}$  the $\ka$-mantle of $V$.
Then there is a ground $W$ such that
$W \subseteq \overline{W}$.
\end{lemma}

For a class $C \subseteq V$ and an ordinal $\alpha$,
let $C_\alpha=C \cap V_\alpha $, the set of all elements of $C$ with rank $<\alpha$.

\begin{lemma}\label{2.7}
Let $\ka$ be a cardinal and $\overline{W}$ the $\ka$-mantle of $V$.
For an inaccessible $\theta>\ka$,
let $\overline{W}^{V_\theta}$ be the $\ka$-mantle of $V_\theta$,
that is, $\overline{W}^{V_\theta}$ is the intersection of all $\ka$-grounds of $V_\theta$.
\begin{enumerate}
\item If $\theta$ is inaccessible $>\ka$ and $W \subseteq V$ is a $\ka$-ground of $V$,
then $W_\theta$ is a $\ka$-ground of $V_\theta$.
\item $\overline{W}^{V_\theta} \subseteq \overline{W}_\theta$
for every inaccessible $\theta>\ka$.
\item Suppose there are proper class many inaccessible cardinals.
Then there is $\alpha>\ka$ such that for every inaccessible cardinal $\theta>\alpha$,
we have $\overline{W}^{V_\theta}=\overline{W}_\theta$.
\end{enumerate}
\end{lemma}

\begin{proof}
(1) is easy, and (2) easily follows from (1).


(3).
Suppose not.
Then, by (2), 
the family ${C}=\{\theta>\ka \mid$ $\theta$ is inaccessible, $\overline{W}^{V_\theta} \subsetneq \overline{W}_\theta \}$
forms a proper class.
For $\theta \in {C}$,
there is a $\ka$-ground $M^\theta$ of $V_\theta$
with $\overline{W}_\theta \nsubseteq M^\theta $.
Fix a poset $\bbP^\theta \in (M^{\theta})_\ka$ and
an $(M^\theta, \bbP^\theta)$-generic $G^\theta$
such that $V_\theta=M^\theta[G^\theta]$.
By Fact \ref{cov. and approx. of forcing ext.}, 
$M^\theta$ has the $\ka$-covering and the $\ka$-approximation
properties for $V_\theta$, and $\ka^+=(\ka^+)^{V_\theta}=(\ka^+)^{M^\theta}$.
Since ${C}$ is a proper class,
there are a poset $\bbP \in V_\ka$,
a filter $G \subseteq \bbP$, and $P \subseteq \p(\ka)$
such that the family 
${C}'=\{\theta \in {C} \mid 
M^{\theta} \cap \p(\ka)=P$, $\bbP^\theta=\bbP$, $G^\theta=G\}$
forms a proper class.
Take $\theta_0$, $\theta_1$ from ${C}'$ with $\theta_0<\theta_1$.
Then $(M^{\theta_1})_{\theta_0}$ is a model of $\ZFC$, and 
$(M^{\theta_1})_{\theta_0} \subseteq V_{\theta_0}$ has 
the $\ka$-covering and the $\ka$-approximation properties for $V_{\theta_0}$.
By applying Fact \ref{uniqueness}, we have 
$(M^{\theta_1})_{\theta_0}=M^{\theta_0}$.
Hence the sequence $\seq{M^{\theta} \mid \theta \in {C}'}$ is coherent,
that is,
$(M^{\theta_1})_{\theta_0}=M^{\theta_0}$ for every $\theta_0<\theta_1$ from ${C}'$.
Then $M=\bigcup_{\theta \in {C}'} M^\theta$ is transitive, closed under
the G\"odel's operations, and almost universal, hence $M$ is a model of $\ZFC$
(see e.g. Theorem 13.9 in Jech \cite{Jech}). 
Moreover $M[G]=V$ because $M^{\theta}[G]=V_\theta$ for every $\theta \in {C}'$,
so $M$ is a $\ka$-ground of $V$.
Therefore we have $\overline{W} \subseteq M$, and 
$\overline{W}_\theta \subseteq M_\theta=M^{\theta}$ for every
$\theta \in {C}'$,
this is a contradiction.
\end{proof}

\section{The proof}
We start the proof of Theorem \ref{mainthm}.

\begin{proof}
Let $\overline{W}$ be the $\ka$-mantle of $V$.
We prove that $\overline{W}$ is the mantle of $V$, this provides Theorem \ref{mainthm};
By Lemma \ref{2.6},
there is a ground $W$ with $W \subseteq \overline{W}$.
Clearly $\bbM \subseteq W \subseteq \overline{W}=\bbM$, hence
$\bbM=\overline{W}=W$ is a ground of $V$.

If not, by Lemma \ref{2.6}, there is a ground $W$ of $V$ with $W \subsetneq \overline{W}$.
Fix a large inaccessible cardinal $\la>\ka$ such that
$W$ is a $\la$-ground of $V$ and $W_\la \subsetneq \overline{W}_\la$.
$W_\la$ and $V_\la$ are transitive models of $\ZFC$.
By Lemma \ref{2.7}, we can find an inaccessible $\theta>\la$
such that $\overline{W}^{V_\theta}=\overline{W}_\theta$,
where $\overline{W}^{V_\theta}$ is the $\ka$-mantle of $V_\theta$.

Take an elementary embedding $j:V_{\theta+1} \to V_{j(\theta)+1}$ such that the
critical point of $j$ is $\ka$ and $\theta<j(\ka)$.
$j(\theta)$ is inaccessible, so $V_{j(\theta)}$ and $W_{j(\theta)}$ are 
transitive models of $\ZFC$.
By the elementarity of $j$,
the set $j(\overline{W}^{V_\theta})$ is the $j(\ka)$-mantle of $j(V_\theta)=V_{j(\theta)}$.
By Lemma \ref{2.7}, $W_{j(\theta)}$ is a $\la$-ground of $V_{j(\theta)}$,
hence 
$j(\overline{W}^{V_\theta}) \subseteq W_{j(\theta)}$.

Fix a sequence $\vec{S}=\seq{S_\alpha \mid \alpha<\la} \in W$
of pairwise disjoint sets such that each $S_\alpha$ is a stationary subset of $\la \cap \mathrm{Cof}(\om)^W$ in $W$.
Since $V$ is a $\la$-c.c. forcing extension of $W$,
each $S_\alpha$ is stationary in $\la \cap \mathrm{Cof}(\om)^V$ in $V$ as well.
Let $\seq{S^*_\alpha \mid \alpha<j(\la)}=j(\vec{S}) \in  V_{j(\theta)}$.
By a well-known argument by Solovay,
we have that
$j``\la=\{\alpha<\sup(j``\la) \mid 
S^*_\alpha \cap \sup(j``\la)$ is stationary in $\sup(j``\la)\}$
(e.g., see Theorem 14 in Woodin-Davis-Rodorigues \cite{WDR}).
\begin{claim}
$j``\la \in W_{j(\theta)}$.
\end{claim}
\begin{proof}[Proof of the claim]
Since $\vec{S} \in W_{\theta} \subseteq \overline{W}_{\theta}
=\overline{W}^{V_\theta}$,
we have $j(\vec{S}) \in j(\overline{W}^{V_\theta}) \subseteq W_{j(\theta)}$.
$V_{j(\theta)}$ is a $\la$-c.c. forcing extension of $W_{j(\theta)}$.
Hence for each set $S \subseteq \sup(j``\la)$ with $S \in W_{j(\theta)}$,
the stationarity of $S$ is absolute between $V_{j(\theta)}$ and $W_{j(\theta)}$.
This means that
for every $\alpha<\sup(j``\la)$, we have that $\alpha \in j``\la$
if and only if $S^*_\alpha \cap \sup(j``\la)$ is stationary in $\sup(j``\la)$ in $W_{j(\theta)}$.
Thus we have
$j``\la \in W_{j(\theta)}$.
\end{proof}

Finally we claim that $W_\la =\overline{W}_\la$, which yields the contradiction.
The inclusion $W_\la \subseteq \overline{W}_\la$ is trivial.
For the converse,
we shall prove $\overline{W}_\alpha \subseteq W_\alpha$ by induction on $\alpha<\la$.
Since the critical point of $j$ is $\ka$,
we have $j(\overline{W}^{V_\theta})_\ka=(\overline{W}^{V_\theta})_\ka=
\overline{W}_\ka$.
Since $j(\overline{W}^{V_\theta}) \subseteq W_{j(\theta)}$, we have
$\overline{W}_\ka \subseteq W_\ka$.
Take $\alpha$ with $\ka \le \alpha<\la$,
and suppose $\overline{W}_\alpha \subseteq W_\alpha$ (so $\overline{W}_\alpha=W_\alpha$).
To show that $\overline{W}_{\alpha+1} \subseteq W$, 
take $X \in \overline{W}_{\alpha+1}$.
Since $\overline{W}$ is transitive, we have $X \subseteq \overline{W}_\alpha$.
$X \in \overline{W}_{\alpha+1} = (\overline{W}^{V_\theta})_{\alpha+1}
\subseteq \overline{W}^{V_\theta}$,
hence $j(X) \in j(\overline{W}^{V_\theta})$.
We know $\overline{W}_\alpha=W_\alpha \in W_\la$ and $W_\la$ is a model of $\ZFC$,
hence there is $\gamma \in W_\la$ and a bijection $f:\gamma \to \overline{W}_\alpha$
with $f \in W_\la$. $W_\la \subseteq \overline{W}_\la$,  so $f \in \overline{W}_\la=(\overline{W}^{V_\theta})_\la
\subseteq \overline{W}^{V_\theta}$,
and $j(f) \in j(\overline{W}^{V_\theta})$.
$j``\la \in W_{j(\theta)}$ and $j(f) \in j(\overline{W}^{V_\theta}) \subseteq W_{j(\theta)}$,
hence $j(f)``(j``\gamma)=j``\overline{W}_\alpha \in W_{j(\theta)}$.
Now $j(X) \in j(\overline{W}^{V_\theta}) \subseteq W_{j(\theta)}$,
thus $j``X=j(X) \cap j``\overline{W}_\alpha \in W_{j(\theta)}$.
Let $\pi \in W_{j(\theta)}$ be the collapsing map of $j``\overline{W}_\alpha$.
Then $\pi``(j``\overline{W}_\alpha)=\overline{W}_\alpha$, and
$X=\pi``(j``X) \in W_{j(\theta)}$, so we have $X \in W_{\alpha+1}$.
\end{proof}
We conclude this paper by asking the following natural question:
\begin{question}
Let $\ka$ be an extendible cardinal.
Is there a poset $\bbP \in \bbM$ of size $<\ka$
and a $(\bbM, \bbP)$-generic $G$ with
$V=\bbM[G]$?
\end{question}
This question is equivalent to the destructibility of extendible cardinals by non-small forcings:
\begin{question}
Let $\ka$ be a cardinal,
and $\bbP$ a poset such that 
for every $p \in \bbP$, the suborder $\{q \in \bbP \mid q \le p\}$
is not forcing equivalent to a poset of size $<\ka$.
Does $\bbP$ necessarily force that ``$\ka$ is not extendible''?
\end{question}

The referee pointed out  that this question might be related to the following result.
See Sargsyan-Schindler \cite{SS} for the definitions.
\begin{fact}[\cite{SS}]
Let $M_{\mathrm{SW}}$ be the least iterable inner model with a strong cardinal above a Woodin cardinal.
If $\ka$ is a strong cardinal of $M_{\mathrm{SW}}$,
then the $\ka$-mantle of $M_{\mathrm{SW}}$ is the smallest ground of
$M_{\mathrm{SW}}$ via some $\ka^+$-c.c. poset of size $\ka^{++}$,
while $M_{\mathrm{SW}}$ cannot be a forcing extension of its $\ka$-mantle via a poset of size $<\ka$.
\end{fact}

\printindex
\end{document}